%% file: ms.tex
\documentclass[11pt, a4paper]{amsart}

\usepackage[margin=0.75in]{geometry}
\usepackage{amsfonts}
\usepackage{amssymb}
\usepackage{amsthm}
\usepackage{amsaddr}
\usepackage{amsmath}
\usepackage{mathtools}
\usepackage{thmtools, thm-restate}

\usepackage{chngcntr}
\usepackage{booktabs}

\usepackage{subcaption}

\newtheorem{theorem}{Theorem}[section]
\newtheorem{proposition}[theorem]{Proposition}
\newtheorem{lemma}[theorem]{Lemma}

\newtheorem{example}[theorem]{Example}

\counterwithin{figure}{section}
\counterwithin{table}{section}
\counterwithin{equation}{section}
\counterwithin{theorem}{section}
\counterwithin{assumption}{section}

\newcommand{\mc}{\mathcal}
\newcommand{\mb}{\mathbb}
\renewcommand{\dim}{\operatorname{dim}}
\newcommand{\dpair}[2]{\left\langle#1,#2\right\rangle}
\newcommand{\rank}[1]{\operatorname{rank}\left(#1\right)}
\newcommand{\clos}[2]{\overline{#2}^{#1}}
\newcommand{\linspan}[1]{\operatorname{span}\left\{#1\right\}}

\newcommand{\bs}{\boldsymbol}
\renewcommand{\r}{\bs r}
\newcommand{\T}{\mc T_{\r}}
\newcommand{\id}{\operatorname{id}}
\newcommand{\R}{\mathbb R}

\newcommand{\N}{\mathbb N}
\newcommand{\Z}{\mathbb Z}

\newcommand{\Umin}{U^{\min}}
\newcommand{\mix}{\text{mix}}
\newcommand{\HS}{\text{HS}}

\newcommand{\cref}{\ref}
\newcommand{\Cref}{\ref}

\begin{document}

\title{Singular Value Decomposition
       in Sobolev Spaces: Part II}

\author{Mazen Ali and Anthony Nouy}

\address{M. Ali: Ulm University, Institute for Numerical Mathematics}
\email{mazen.ali@uni-ulm.de}
\address{A. Nouy: Centrale Nantes, LMJL (UMR CNRS 6629)}
\email{anthony.nouy@ec-nantes.fr}

\begin{abstract}
Under certain conditions, an element of a tensor product space
can be identified with a compact operator and the singular value
decomposition (SVD) applies to the latter.
These conditions are not fulfilled in
Sobolev spaces.
In the previous part of this work (part I),
we introduced some preliminary notions in the theory of tensor product spaces.
We analyzed low-rank approximations in $H^1$ and the error of the
SVD performed in the ambient $L^2$ space.

In this work (part II), we continue
by considering variants of the SVD in norms stronger than the $L^2$-norm.
Overall and, perhaps surprisingly, this leads to a more difficult
control of the $H^1$-error. We briefly consider an isometric embedding
of $H^1$ that allows direct application of the SVD to $H^1$-functions.
Finally,
we provide a few numerical examples
that support our theoretical findings.
\end{abstract}



\maketitle

\input{Sections/intro}
\input{Sections/h1svd}
\input{Sections/alt}
\input{Sections/numexp}
\input{Sections/conclusion}
\input{Sections/ack}
\input{Sections/appendix}

\bibliographystyle{acm}
\bibliography{literature}

\end{document}

%% file: Sections/intro.tex
\section{Introduction}
A function $u$ in the tensor product $H=H_1\otimes H_2$ of two Hilbert spaces
$H_1$ and $H_2$ may be identified with a compact operator
$T_u:H_2\rightarrow H_1$. This identification is possible when the norm
on $H$ is not weaker than the injective norm, i.e., in a certain sense
the norm on
$H$ is compatible with the norms on $H_1$ and $H_2$. In such a case
we can decompose $u$ as
\begin{align}\label{eq:introsvd}
    u=\sum_{k=1}^\infty\sigma_k\psi_k\otimes\phi_k,
\end{align}
for a non-negative non-increasing sequence $\{\sigma_k\}_{k\in\N}$ and
orthonormal
systems
$\{\psi_k\}_{k\in\N}\subset H_1$ and
$\{\phi_k\}_{k\in\N}\subset H_2$. This is the well known
\emph{singular value decomposition (SVD)} and it provides low-rank
approximations via
\begin{align*}
   \inf_{\rank{v_r}\leq r} \left\|u-u_r\right\|^2_H=
   \left\|u-\sum_{k=1}^r\sigma_k\psi_k\otimes\phi_k\right\|^2_H=
   \sum_{k=r+1}^\infty\sigma_k^2,
\end{align*}
if, e.g., $\|\cdot\|_{H}$ is the canonical norm on a Hilbert tensor space.

If $H$ is the Sobolev space of once weakly differentiable functions,
the above assumption is not satisfied and there is no SVD for a function
$u\in H^1$ in general. The focus of this work is to explore variants of the
SVD in different ambient spaces in the $H^1$-norm. In part I, we showed
that low-rank approximations in the Tucker format in $H^1$ exist. More precisely,
\begin{theorem}
    $\T\left({}_a\bigotimes_{j=1}^d H^1(\Omega_j)\right)$ is
    weakly closed and therefore proximinal in $H^1(\Omega)$.
\end{theorem}
We also showed under which conditions
\begin{align*}
        u\in{}_{\|\cdot\|_1}\bigotimes_{j=1}^d\Umin_j(u),
\end{align*}
i.e., $u$ belongs to the tensor product of its minimal subspaces.
Finally, we analyzed the $H^1$-error of the $L^2$-SVD for a general order
$d\geq 2$.

In this part, we consider the intersection space structure of $H^1$
\begin{align*}
    H^1(\Omega_1\times\Omega_2)=\bigg(H^1(\Omega_1)\otimes L^2(\Omega_2)\bigg)\cap
    \bigg(L^2(\Omega_1)\otimes H^1(\Omega_2)\bigg)=:H^{(1,0)}\cap H^{(0,1)}.
\end{align*}
We analyze the $H^1$-error of the $H^{(1,0)}$-
and $H^{(0,1)}$-SVDs.
We also consider an isometric embedding of
$H^1$ into a space which allows the direct application of the SVD.

The paper is organized as follows.
In Section \cref{sec:approx},
we consider the $H^{(1,0)}$-
and $H^{(0,1)}$-SVDs
and generalizations to higher dimensions.
In Section \cref{sec:alt}, we consider the SVD in higher-dimensional
spaces of mixed smoothness, exponential sum approximations and
an isometric embedding of $H^1$ that allows a direct application
of the SVD.
We conclude in Section \cref{sec:numexp}
with some simple numerical
experiments with different types of low-rank approximations.
Section \cref{sec:concl} summarizes the results of part I and part II.

%% file: Sections/h1svd.tex
\section{SVD in $H^{(1,0)}$ and $H^{(0, 1)}$}\label{sec:approx}
Before we proceed with analyzing SVDs in $H^{(1,0)}$ and
$H^{(0, 1)}$,
we consider the corresponding singular values and compare them to
$L^2$-singular values.

\subsection{$H^{(1,0)}$ and $H^{(0, 1)}$ Singular Values}
We consider a function $u\in H^1(\Omega)$ as an element of the intersection
space
$
    u\in H^{(1,0)}\cap H^{(0,1)}.
$
We first consider $u$ as a Hilbert Schmidt operator
$u:L^2(\Omega_2)\rightarrow H^1(\Omega_1)$ defined by
\begin{align*}
    u[w]=\int_{\Omega_2}u(\cdot, y)w(y)dy,\quad w\in L^2(\Omega_2).
\end{align*}
The difference to simply viewing $u$ as an $L^2$ integral kernel
arises
when we consider the adjoint $u^*:H^1(\Omega_1)\rightarrow L^2(\Omega_2)$
\begin{align*}
    u^*[v] = \int_{\Omega_1}u(x, \cdot)v(x)+
    \frac{\partial}{\partial x}u(x,\cdot)\frac{d}{dx}v(x)dx,\quad
    v\in H^1(\Omega_1).
\end{align*}
The corresponding left and right singular functions $\psi_k^1\in H^1(\Omega_1)$
and $\phi_k^0\in L^2(\Omega_2)$ are respectively given by
\begin{align*}
    uu^*[\psi_k^1]=\int_{\Omega_2}u(\cdot, y)
    \int_{\Omega_1}u(x, y)\psi_k^1(x)+\frac{\partial}{\partial x}u(x, y)
    \frac{d}{dx}\psi_k^1(x)dxdy = \lambda^{10}_k\psi_k^1,
\end{align*}
and
\begin{align}\label{eq:singfuncs}
    u^*u[\phi_k^0]&=\int_{\Omega_1}u(x,\cdot)
    \int_{\Omega_2}u(x, y)\phi_k^0(y) dydx\notag\\
    &+
    \int_{\Omega_1}\frac{\partial}{\partial x}u(x, \cdot)
    \int_{\Omega_2}\frac{\partial}{\partial x}u(x, y)
    \phi_k^0(y)dydx =\lambda^{10}_k\phi_k^0,
\end{align}
with the corresponding singular values $\sigma_k^{10}=\sqrt{\lambda_k^{10}}$
sorted in decreasing order.
Note that, unlike in the previous subsection, in general
$\phi_k^0\not\in H^1(\Omega_2)$. To guarantee this we would have to require
$u\in H^1_{\text{mix}}(\Omega)$. This means that the sum
$
    u_r = \sum_{k=1}^r\sigma^{10}_k\psi_k^1\otimes\phi_k^0
$
does not make sense in $H^1(\Omega)$ in general, only in
$H^{(1,0)}$.

Similarly, we can interpret $u\in H^1(\Omega)$ as a Hilbert Schmidt operator\\
$u:H^1(\Omega_2)\rightarrow L^2(\Omega_1)$ defined by
\begin{align*}
    u[w]=\int_{\Omega_2}u(\cdot, y)w(y)
    +
    \frac{\partial}{\partial y}u(\cdot, y)\frac{d}{dy}w(y)dy
    ,\quad w\in H^1(\Omega_2),
\end{align*}
with an adjoint $u^*:L^2(\Omega_1)\rightarrow H^1(\Omega_2)$ given by
$
    u^*[v]=\int_{\Omega_1}u(x, \cdot)v(x)dx,\quad v\in L^2(\Omega_1).
$
The corresponding singular functions $\psi_k^0\in L^2(\Omega_1)$ and
$\phi_k^1\in H^1(\Omega_2)$ satisfy
\begin{align*}
    uu^*[\psi_k^0]&=\int_{\Omega_2}u(\cdot, y)
    \int_{\Omega_1}u(x, y)\psi_k^0(x) dxdy\\
    &+
    \frac{\partial}{\partial y}u(\cdot, y)
    \int_{\Omega_1}\frac{\partial}{\partial y}u(x, y)
    \psi_k^0(x)dxdy =\lambda^{01}_k\psi_k^0,
\end{align*}
and
\begin{align}
    u^*u[\phi_k^1]=
    \int_{\Omega_1}u(x, \cdot)
    \int_{\Omega_2}u(x, y)\phi_k^1(y)+\frac{\partial}{\partial y}u(x, y)
    \frac{d}{dy}\phi_k^1(y)dydx = \lambda^{01}_k\phi_k^1,
\end{align}
where $\sigma_k^{01}=\sqrt{\lambda_k^{01}}$ are the corresponding singular
values, sorted in decreasing order.
We make the following immediate observation.
\begin{proposition}\label{prop:relationsingval}
    Let $u\in H^1(\Omega)$ and let
    $\sum_{k=1}^\infty\sigma_k^{00}\psi_k\otimes\phi_k$,
    $\sum_{k=1}^\infty\sigma_k^{10}\psi^1_k\otimes\phi^0_k$ and
    $\sum_{k=1}^\infty\sigma_k^{01}\psi^0_k\otimes\phi^1_k$ be the SVD of
    $u$ interpreted as an element of $H^{(0,0)}$, $H^{(1,0)}$ and
    $H^{(0,1)}$ respectively.
    Then, we have for all $r\geq 0$
    \begin{align*}
        \sum_{k=r+1}^\infty(\sigma_k^{00})^2&\leq
        \sum_{k=r+1}^\infty(\sigma_k^{10})^2\leq
        \sum_{k=r+1}^\infty(\sigma_k^{00})^2\|\psi_k\|_1^2,\\
        \sum_{k=r+1}^\infty(\sigma_k^{00})^2&\leq
        \sum_{k=r+1}^\infty(\sigma_k^{01})^2\leq
        \sum_{k=r+1}^\infty(\sigma_k^{00})^2\|\phi_k\|_1^2.
    \end{align*}
\end{proposition}
\begin{proof}
    The first statement is given by
    \begin{align*}
        \sum_{k=r+1}^\infty(\sigma^{10}_k)^2&=
        \inf_{v\in\mc R_{r}(H^1(\Omega_1){}_a\otimes L^2(\Omega_2))}
        \|u-v\|_{(1,0)}^2\\
        &\geq
        \inf_{v\in\mc R_{r}(L^2(\Omega_1){}_a\otimes L^2(\Omega_2))}
        \|u-v\|_{(0,0)}^2=
        \sum_{k=r+1}^\infty(\sigma^{00}_k)^2,
    \end{align*}
    and
    \begin{align*}
        \sum_{k=r+1}^\infty(\sigma^{10}_k)^2&=
        \inf_{v\in\mc R_{r}(H^1(\Omega_1){}_a\otimes L^2(\Omega_2))}
        \|u-v\|^2_{(1,0)}\leq
         \Big\|u-\sum_{k=1}^{r}\sigma_k^{00}\psi_k\otimes\phi_k
        \Big\|_{(1,0)}^2 \\
        &=
        \sum_{k=r+1}^\infty(\sigma_k^{00})^2\|\psi_k\|_1^2.
    \end{align*}
    Analogously for the second statement.
\end{proof}
Note that the upper bounds in Proposition
\cref{prop:relationsingval} do not necessarily
hold component-wise, i.e., the inequalities
$
    \sigma_k^{10}\leq\sigma_k^{00}\|\psi_k\|_1,
$
do not hold in general. This is due to the fact that when estimating the
injective norm
\[
    \left\|\sum_{k=r+1}^{\infty}\sigma_k^{00}\psi_k\otimes\phi_k
    \right\|_{\vee\left(H^1(\Omega_1), L^2(\Omega_2)\right)},
\]
the functions $\psi_k$ are not orthonormal in $H^1(\Omega_1)$ and
the sequence
$\{\sigma_k^{00}\|\psi_k\|_1\}_{k\in\N}$ is not necessarily
decreasing.

Naturally, we can ask whether we can derive a bound of the sort
\[
    \sum_{k=r+1}^\infty(\sigma_k^{00})^2\|\psi_k\|_1^2\lesssim
    \sum_{k=r+1}^\infty(\sigma_k^{10})^2\gamma(k),
\]
for some sequence $\gamma(k)$.
Though we do not believe this is possible
without further assumptions, we can nonetheless improve
the bounds. This indicates that indeed the quantities
$\sigma_k^{10}$ and $\sigma_k^{00}\|\psi_k\|_1$
are closely related. This
will later be confirmed by numerical observations.

\begin{theorem}\label{thm:singvalsbounds}
    Let $u\in H^1(\Omega)$ and assume the $L^2$-SVD
    $u=\sum_{k=1}^\infty\sigma_k^{00}\psi_k\otimes\phi_k$ converges
    in $H^1(\Omega)$. Then, we have
    \begin{align*}
        \sigma_r^{10}
        &=\|\psi_r^1\|_0^{-1/2}\left(
        \sum_{k=1}^\infty(\sigma_k^{00})^4|\dpair{\psi_r^1}{\psi_k^0}_1|^2
        \right)^{1/4}
        \geq\sigma_r^{00}\left(
        \frac{\dpair{\psi_r^1}{\psi_r}_1}{\|\psi_r^1\|_0}\right)^{1/2},\\
        \sigma_r^{01}
        &=\|\phi_r^1\|_0^{-1/2}\left(
        \sum_{k=1}^\infty(\sigma_k^{00})^4|\dpair{\phi_r^1}{\phi_k^0}_1|^2
        \right)^{1/4}
        \geq\sigma_r^{00}\left(
        \frac{\dpair{\phi_r^1}{\phi_r}_1}{\|\phi_r^1\|_0}\right)^{1/2}.
    \end{align*}
\end{theorem}
\begin{proof}
    We consider the $L^2$-SVD
    $u=\sum_{k=1}^\infty\sigma_k^{00}\psi_k\otimes\phi_k$.
    $u$ is identified with an operator
    $u:L^2(\Omega_2)\rightarrow H^1(\Omega_1)$.
    For any $w\in L^2(\Omega_2)$,\\
$
        u[w] = \sum_{k=1}^\infty\sigma_k^{00}
        \dpair{w}{\phi_k}_0\psi_k,
$
    converges in $H^1(\Omega_1)$ and for any $v\in H^1(\Omega_1)$,
$
        u^*[v]=\sum_{k=1}^\infty\sigma_k^{00}\dpair{v}{\psi_k}_1\phi_k,
$
    convergences in $L^2(\Omega_2)$.
    Thus,
    \begin{align*}
        uu^*[v]&=\sum_{k=1}^\infty\sigma_k^{00}
        \dpair{\sum_{l=1}^\infty\sigma_l^{00}\dpair{v}{\psi_l}_1\phi_l}{
        \phi_k}_0\psi_k\\
        &=
        \sum_{k=1}^\infty\sum_{l=1}^\infty\sigma_k^{00}\sigma_l^{00}
        \dpair{v}{\psi_l}_1\dpair{\phi_l}{\phi_k}_0\psi_k\\
        &=\sum_{k=1}^\infty (\sigma_k^{00})^2
        \dpair{v}{\psi_k}_1\psi_k.
    \end{align*}
    On the other hand, utilizing the $H^{(1,0)}$-SVD of $u$, we have\\
$
        uu^*[v]=\sum_{k=1}^\infty(\sigma_k^{10})^2\dpair{v}{\psi_k^1}\psi_k^1,
$
    and thus
    \begin{align*}
        uu^*[v]=
        \sum_{k=1}^\infty (\sigma_k^{00})^2
        \dpair{v}{\psi_k}_1\psi_k
        =
        \sum_{k=1}^\infty(\sigma_k^{10})^2\dpair{v}{\psi_k^1}\psi_k^1.
    \end{align*}
    Substituting $v=\psi_r^1$, we obtain
$
        \sum_{k=1}^\infty(\sigma_k^{00})^2
        \dpair{\psi_r^1}{\psi_k}_1\psi_k^0=(\sigma_r^{10})^2\psi_r^1,
$
    since $\{\psi_k^1\}_{k\in\N}$ are $H^1(\Omega_1)$-orthonormal.
    Finally, taking the $L^2(\Omega_1)$-norm of both sides and since
    $\{\psi_k\}_{k\in\N}$ are $L^2(\Omega_1)$-orthonormal, we obtain\\
$
        \sum_{k=1}^\infty(\sigma_k^{00})^4|\dpair{\psi_r^1}{\psi_k}|^2
        =(\sigma_r^{10})^4\|\psi_r^1\|_0^2.
$
    The statement for $\sigma_r^{01}$ follows analogously by
    identifying $u$ with an operator from $H^1(\Omega_2)$ to
    $L^2(\Omega_1)$. This completes the
    proof.
\end{proof}

The factors in the bounds in Theorem
\cref{thm:singvalsbounds} reflect how $\psi_r^1$,
normalized in $H^1(\Omega_1)$, scales w.r.t.\ $\psi_r$, normalized in
$L^2(\Omega_1)$. For instance, if $\{\psi_r\}_{r\in\N}$ behaves like Fourier
or wavelet basis,
then $\psi^1_r\sim\|\psi_r\|^{-1}_1\psi_r$. In this case, the right hand
side in Theorem \cref{thm:singvalsbounds} evaluates to
$
        \sigma_r^{00}\left(
        \frac{\dpair{\psi_r^1}{\psi_r}_1}{\|\psi_r^1\|_0}\right)^{1/2}
        \sim\sigma_r^{00}\|\psi_r\|_1.
$
This leads precisely to the upper bound of Proposition
\cref{prop:relationsingval}.
Analogous conclusions hold when considering $\sigma_r^{00}$, $\phi_r^1$ and
$\phi_r$.

Extending the results of this subsection
to $d>2$ using HOSVD
singular values and, e.g., the Tucker format is straightforward. Since we can
consider matricizations w.r.t.\ to each $1\leq j\leq d$ separately, the analysis
effectively reduces to the case $d=2$. Difficulties arise only when considering
simultaneous projections in different components of the tensor product
space. There we have to assume the rescaled singular values converge, as
was done in part I of this work.

\subsection{$H^{(1,0)}$ and $H^{(0, 1)}$ projections}\label{sec:h1l2proj}
Given the singular functions\\
$\{\psi_k^1\}_{k\in\N}$ and
$\{\phi^1_k\}_{k\in\N}$ associated with $H^{(1,0)}$ and
$H^{(0,1)}$ SVDs of $u$ respectively, we consider the finite
dimensional subspaces
\begin{align}\label{eq:Br}
    B^1_r&:=\linspan{\psi_k^1:1\leq k\leq r}\subset\Umin_1(u),\\
    B^2_r&:=\linspan{\phi_k^1:1\leq k\leq r}\subset\Umin_2(u),\notag
\end{align}
and the corresponding $H^1$-orthogonal projections
\begin{align}\label{eq:h1l2proj}
    P_r:H^1(\Omega_1)\rightarrow B^1_r,\quad
    Q_r:H^1(\Omega_2)\rightarrow B^2_r.
\end{align}
The tensor product $P_r\otimes Q_r$ is well defined on
$H^1(\Omega_1)\otimes_a H^1(\Omega_2)$, and on this space it holds
\begin{align}\label{eq:comm}
    P_r\otimes Q_r=(P_r\otimes\id_2)(\id_1\otimes Q_r)
    =(\id_1\otimes Q_r)(P_r\otimes\id_2).
\end{align}
However, the interpretation is problematic when considering
$P_r\otimes Q_r$ on the closure of
$H^1(\Omega_1)\otimes_a H^1(\Omega_2)$ . Take, e.g., the projection
$P_r\otimes\id_2$. This is an orthogonal projection on $H^{(1,0)}$
and we have
$
    (P_r\otimes\id_2)u=\sum_{k=1}^r\sigma_k^{10}\psi_k^1\otimes\phi_k^0.
$
But in general
$
    (P_r\otimes\id_2)u\not\in H^{(0,1)},
$
unless $u\in H^1(\Omega_1)\otimes_a H^1(\Omega_2)$. Thus, the subsequent
application $\id_1\otimes Q_r$ does not necessarily make sense and is
not continuous.

Notice the difference with the projections $P_{r}^1$ and $P_r^2$ for
$L^2$-SVD from part I (for $d=2$ and $r_1=r_2=r$). First, we had
$
    P_r^1\otimes\id_2 u=\id_1\otimes P_r^2 u=P^1_r\otimes P^2_r u,
$
since both the left and right projections already give the best
rank $r$ approximation in $L^2$.
Second, we required only $L^2$-orthogonality, thus preserving $H^1$-regularity in the image. Thus,
$P^1_r\otimes P^2_r$ made sense on $H^1(\Omega)$, although the sequence
of projections does not necessarily converge in $H^1(\Omega)$.
To that end, we had
to additionally assume in part I the
convergence of the rank-$r$ approximations $u_r$,
or convergence of the rescaled $L^2$-singular values.

In the present case, although we obtain optimality in the stronger
$\|\cdot\|_{(1,0)}$-norm, we lose convergence or possibly even boundedness in
the $\|\cdot\|_{(0,1)}$-norm. Thus, we can ask ourselves if $P_r$
is bounded from $L^2(\Omega_1)$ to $L^2(\Omega_1)$, i.e., if
$
    P_r\in\mc L\left(L^2(\Omega_1), L^2(\Omega_1)\right)$?

Specifically, what are the minimal assumptions - if any - that we require in
order to achieve this? The next example shows that indeed even for simple
projections this property is not guaranteed.

\begin{example}\label{ex:discproj}
    Let $\Omega_1=(0,1)$ and consider the space $H_0^1(0,1)$. We know
    $H_0^1(0,1)\hookrightarrow C(0,1)$. Consider $g\in H_0^1(0,1)
    \rightarrow\R$ defined by
$
        g[f]:=f(0.5),$ $\forall f\in H_0^1(0,1).
$
    Clearly, $g$ is a linear functional. Moreover, since any such $f$ is
    absolutely continuous, $g$ is bounded in the $\|\cdot\|_1$-norm.
    Thus, $g\in (H_0^1(0,1))^*$. By the Riesz representation theorem, there
    exists a unique $\tilde{g}\in H_0^1(0,1)$, such that
    $g[f]=\dpair{f}{\tilde{g}}_1$, for all $f\in H_0^1(0,1)$.

    Define the one dimensional subspace
    $U=\linspan{\tilde{g}}$.
    The corresponding $H^1$-orthogonal projection $P$ is given by
$
        Pv=\|\tilde{g}\|_1^{-2}\dpair{v}{\tilde g}_1\tilde{g}
       =\|\tilde{g}\|_1^{-2}v(0.5)\tilde{g}.
$
    Consider the sequence
    \begin{align*}
        v_n(x):=
        \begin{cases}
            1+(n+1)(x-0.5),&\;\text{if}\quad 0.5-\frac{1}{n+1}\leq x<0.5,\\
            1+(n+1)(0.5-x),&\;\text{if}\quad 0.5\leq x\leq0.5+\frac{1}{n+1},\\
            0,&\;\text{otherwise},
        \end{cases}
    \end{align*}
    $n\in\N$. Clearly, $v_n\in H_0^1(0, 1)$ for any $n\in\N$,
    $\|v_n\|_0\leq\sqrt{\frac{2}{n+1}}\longrightarrow 0$,
    and
    $Pv_n=\|\tilde{g}\|_1^{-2}\tilde{g}\neq 0$, for all $n\in\N$.
    Thus, $P$ can not be continuous in $L^2$.
\end{example}

A closer look at the preceding example shows that such a function
$\tilde{g}\in H_0^1(0, 1)$ differentiated twice yields the delta distribution.
Therefore, it can not be in $H^2(\Omega_1)$. On the other hand, if
the function has $H^2$-regularity, as the next statement shows, we can
indeed obtain boundedness in $L^2$.

\begin{lemma}\label{lemma:prbound}
    Let $u\in H^1(\Omega)$. In addition, assume the second
    unidirectional derivatives of $u$ exist in the distributional sense and are
    bounded, i.e.,
$
        \left\|\frac{\partial^2}{\partial x^2}u\right\|_0<\infty,$ $
        \left\|\frac{\partial^2}{\partial y^2}u\right\|_0<\infty.
$
    Finally, assume $u$ satisfies either zero Dirichlet or zero Neumann boundary
    conditions. Then, the projections defined in \eqref{eq:h1l2proj}
    can be bounded as
    \begin{align*}
        \|P_r v\|_1&\leq\sqrt{2}\|v\|_0\left(
        \sum_{k=1}^r\|\psi_k^1\|_0^2+\left\|
        \frac{d^2}{dx^2}\psi_k^1\right\|_0^2\right)^{1/2},\\
        \|Q_r w\|_1&\leq\sqrt{2}\|w\|_0\left(
        \sum_{k=1}^r\|\phi_k^1\|_0^2+\left\|
        \frac{d^2}{dy^2}\phi_k^1\right
        \|_0^2\right)^{1/2}.
    \end{align*}
\end{lemma}
\begin{proof}
    One can easily verify that $\psi_k^1$ and $\phi_k^1$ are
    twice weakly differentiable for any $k\in\N$.
    For any $v\in H^1(\Omega_1)$, we can write
$
        P_rv=\sum_{k=1}^r\dpair{v}{\psi_k^1}_1\psi_k^1.
$
    The coefficients can be written as
    \begin{align*}
        \dpair{v}{\psi_k^1}_1&=
        \int_{\Omega_1}v(x)\psi_k^1(x)dx+
        \int_{\Omega_1}\frac{d}{dx}v(x)\frac{d}{dx}\psi_k^1(x)dx\\
        &=
        \int_{\Omega_1}v(x)\psi_k^1(x)dx-
        \int_{\Omega_1}v(x)\frac{d^2}{dx^2}\psi_k^1(x)dx,
    \end{align*}
    where the boundary term vanishes due to the boundary conditions.
    Thus, we get
    \begin{align*}
        \|P_rv\|_1^2&=\sum_{k=1}^r|\dpair{v}{\psi_k^1}_1|^2
        \leq\sum_{k=1}^r\left(\|v\|_0\|\psi_k^1\|_0+
        \|v\|_0\left\|\frac{d^2}{dx^2}\psi_k^1\right\|_0\right)^2\\
        &\leq2\|v\|_0^2\sum_{k=1}^r\left(\|\psi_k^1\|^2_0+
        \left\|\frac{d^2}{dx^2}\psi_k^1\right\|^2_0\right).
    \end{align*}
    Analogously for $Q_r$. This completes the proof.
\end{proof}

Note that in principle the assumption on the boundary conditions can be
replaced or avoided, as long as we can estimate the appearing boundary term.
The assumption can be avoided entirely
by using an estimate for the $L^\infty$-norm
via the Gagliardo-Nirenberg inequality, although this would yield a crude
estimate and dimension dependent regularity requirements.

Under the conditions of Lemma \cref{lemma:prbound}, we can assert that
$P_r\otimes Q_r$
is indeed continuous. Since $\|\cdot\|_{(0,1)}$ is a uniform crossnorm
\begin{align*}
    \|P_r\otimes Q_r\|_{(0,1)}=\|P_r\|_0\|Q_r\|_1\leq
    \sqrt{2}\left(\sum_{k=1}^r\|\psi_k^1\|_0^2+\left\|
    \frac{d^2}{dx^2}\psi_k^1\right\|_0^2\right)^{1/2},
\end{align*}
and similarly for $\|\cdot\|_{(1,0)}$. Thus,
$
    P_r\otimes Q_r\in\mc L\left(H^1(\Omega), H^1(\Omega)\right).
$
By density, we can uniquely extend $P_r\otimes Q_r$ onto $H^1(\Omega)$
and the identity \eqref{eq:comm} holds.

One might argue that requiring $P_r$ and $Q_r$
to be continuous in $L^2$ is
unnecessary, since we only need that the mappings
$
    P_r\otimes\id_2:
    H^1(\Omega)
    \rightarrow
    H^{(0,1)},
$
and $\id_1\otimes Q_r$ are continuous. The following example shows
that indeed $P_r\otimes\id_2$
need not be continuous even on elementary tensor
products, if $P_r$ is not $L^2$
continuous.

\begin{example}\label{ex:pidnotc}
    Take $P$ to be the projection from Example \cref{ex:discproj}. Consider the same
    sequence $\{v_n\}_{n\in\N}\subset H_0^1(0,1)$ as in Example \cref{ex:discproj}.
    Take another sequence $w_n\in H^1_0(0,1)$ as
    \begin{align*}
        w_n(y):=
        \begin{cases}
            (n+1)^{-1/2}+(n+1)^{1/2}(y-0.5),&\;\text{if}\quad 0.5-n^{-1}\leq y < 0.5,\\
            (n+1)^{-1/2}+(n+1)^{1/2}(0.5-y),&\;\text{if}\quad 0.5 \leq y <0.5+n^{-1},\\
            0,&\;\text{otherwise}.
        \end{cases}
    \end{align*}
    Then,
    \begin{align*}
        \|w_n\|_0^2&\leq 2(n+1)^{-2},\\
        \|w_n\|_1^2&\leq 2(n+1)^{-2}+2(n+1)(n+1)^{-1}=2(n+1)^{-2}+2,\\
        \|w_n\|_1^2&\geq 2.
    \end{align*}
    Thus, since $\|\cdot\|_{(0,1)}$ is a crossnorm
    \begin{align*}
        \|(P\otimes\id_2)(v_n\otimes w_n)\|_{(0,1)}=
        \|Pv_n\|_0\|w_n\|_1\geq 2\|Pv_1\|_0>0,\quad\forall n\in\N.
    \end{align*}
    On the other hand
    \begin{align*}
        &\|v_n\otimes w_n\|_1^2\leq
        \|v_n\otimes w_n\|_{(10)}^2+
        \|v_n\otimes w_n\|_{(01)}^2
        =
        \|v_n\|^2_1\|w_n\|^2_0+
        \|v_n\|^2_0\|w_n\|^2_1\\
        &\leq
        [2(n+1)^{-1}+2(n+1)][2(n+1)^{-2}]+
        [2(n+1)^{-1}][2(n+1)^{-2}+2]\\
        &=[4(n+1)^{-3}+4(n+1)^{-1}]+[4(n+1)^{-3}+4(n+1)^{-1}]
        \longrightarrow 0.
    \end{align*}
    Hence, $P\otimes\id_2$ is not continuous on $H^1(\Omega)$
    even on $H^1(\Omega_1)\otimes_a H^1(\Omega_2)$.
\end{example}

To summarize our findings, let us define the finite dimensional subspaces
$
    W^1_r:=\linspan{\psi_k^0:1\leq k\leq r},$ $
    W^2_r:=\linspan{\phi_k^0:1\leq k\leq r},
$
Under the assumptions of Lemma \cref{lemma:prbound},
$W_r^1\subset H^1(\Omega_1)$ and $W_r^2\subset H^1(\Omega_2)$.
This can also be observed by,
e.g., considering \eqref{eq:singfuncs} and integrating the second term by parts.
We can estimate the $H^1$ error as follows.

\begin{theorem}\label{thm:h1l2bound}
    Let the assumptions of Lemma \cref{lemma:prbound} be satisfied. Moreover,
    define the constants
    \begin{align*}
        L(r):=\sup_{v\in B^1_r}\frac{\|v\|_1}{\|v\|_0}
        \sup_{v\in B^1_r}\frac{\|v\|_2}{\|v\|_1},\quad
        R(r):=\sup_{w\in B^2_r}\frac{\|w\|_1}{\|w\|_0}
        \sup_{w\in B^2_r}\frac{\|w\|_2}{\|w\|_1}.
    \end{align*}
    Then, the projection error is bounded as
    \begin{align*}
        &\frac{1}{\sqrt{2}}
        \left(\sum_{k=r+1}^\infty
        (\sigma_k^{10})^2+(\sigma_k^{01})^2\right)^{1/2}
        \leq \left\|u-(P_r\otimes Q_r)u\right\|_1\\
        &\leq
        \left(\sum_{k=r+1}^\infty(1+2r^2R(r)^2)(\sigma_k^{10})^2+
        (1+2r^2L(r)^2)(\sigma_k^{01})^2\right)^{1/2}.
    \end{align*}
\end{theorem}
\begin{proof}
    For the lower bound observe first that
    \begin{align*}
        \|u-(P_r\otimes Q_r)u\|^2_1\geq\frac{1}{2}\left(
        \|u-(P_r\otimes Q_r)u\|^2_{(1,0)}+
        \|u-(P_r\otimes Q_r)u\|^2_{(0,1)}\right).
    \end{align*}
    Since $P_r\otimes\id_2 u$ is the optimal rank $r$ approximation in the
    $\|\cdot\|_{(1,0)}$-norm, we can further estimate
    \begin{align*}
        \|u-(P_r\otimes Q_r)u\|^2_{(1,0)}\geq
        \|u-(P_r\otimes\id_2)u\|^2_{(1,0)}
        =\sum_{k=r+1}^\infty(\sigma_k^{10})^2,
    \end{align*}
    and similarly for $\id_1\otimes Q_r$. This gives the lower bound.

    For the upper bound, since $P_r\otimes\id_2$ is orthogonal in the
    $\|\cdot\|_{(1,0)}$-norm, we get
    \begin{align*}
        &\|u-(P_r\otimes Q_r)u\|^2_{(1,0)}\\
        &=
        \|u-(P_r\otimes\id_2)u\|^2_{(1,0)}+
        \|P_r\otimes\id_2[u-(\id_1\otimes Q_r)u]\|^2_{(1,0)}.
    \end{align*}
    To estimate the latter term, recall that
    $(P_r\otimes\id_2)u\in B_r^1\otimes_aW^2_r$.
    Thus, we can find some $\{v_i\}_{i=1}^r$ in $B^1_r$ and
    $\{w_i\}_{i=1}^r$ in $W^2_r$ such that
    \begin{align}\label{eq:errep}
        e_r:=P_r\otimes\id_2[u-(\id_1\otimes Q_r)u]
        =\sum_{k=1}^r v_i\otimes w_i.
    \end{align}
    Thus, we estimate further
    \begin{align*}
        \|e_r\|_{(1,0)}^2&\leq\left(\sum_{k=1}^r\|v_i\|_1\|w_i\|_0\right)^2
        \leq
        \left(\sum_{k=1}^rL_1(r)\|v_i\|_0\|w_i\|_1\right)^2,
    \end{align*}
    where
$
        L_1(r):=\sup_{v\in B^1_r}\frac{\|v\|_1}{\|v\|_0}.
$
    Taking the infimum over all representations \eqref{eq:errep} of $e_r$,
    we obtain
$
        \|e_r\|_{(1,0)}^2\leq L_1(r)^2\|e_r\|^2_{\wedge(0,1)},
$
    where $\|\cdot\|_{\wedge(0,1)}$ is the projective norm on
    $L^2(\Omega_1)\otimes_a H^1(\Omega_2)$. Let
    $\{\sigma^e_{k}\}_{k=1}^r$ denote the singular values of
    $e_r:H^1(\Omega_2)\rightarrow L^2(\Omega_1)$. Then, since the projective
    norm corresponds to the nuclear norm of the operator $e_r$ (see also
    \cite[Remark 4.116]{HB})
    \begin{align*}
        \|e_r\|^2_{\wedge(0,1)}\leq\left(\sum_{k=1}^r\sigma_k^e\right)^2
        \leq r\sum_{k=1}^r(\sigma_k^e)^2=r\|e_r\|^2_{(0,1)}.
    \end{align*}
    In summary,
$
        \|e_r\|_{(1,0)}^2\leq L_1(r)^2r\|e_r\|_{(0,1)}^2.
$
    Finally, to bound $P_r\otimes\id_2$, we apply Lemma \cref{lemma:prbound}
    \begin{align*}
        \|P_r\otimes\id_2\|_{(0,1)}&\leq
        \sqrt{2}\left(
        \sum_{k=1}^r\|\psi_k^1\|_0^2+\left\|
        \frac{d^2}{dx^2}\psi_k^1\right\|_0^2\right)^{1/2}\\
        &\leq
        \sqrt{2}\left(
        \sum_{k=1}^r\|\psi_k^1\|_0^2+
        \left\|
        \frac{d}{dx}\psi_k^1\right\|_0^2+
        \left\|
        \frac{d^2}{dx^2}\psi_k^1\right\|_0^2\right)^{1/2}\\
        &=
        \sqrt{2}\left(
        \sum_{k=1}^r\frac{\|\psi_k^1\|^2_2}{\|\psi_k^1\|^2_1}\right)^{1/2}
        \leq\sqrt{2}\sqrt{r}L_2(r),
    \end{align*}
    since $\psi_k^1$ are $H^1$ normalized and
    $L_2(r):=\sup_{v\in B^1_r}\frac{\|v\|_2}{\|v\|_1}$.
    Thus,
    \begin{align*}
        \|u-(P_r\otimes Q_r)u\|_{(1,0)}^2&\leq
        \sum_{k=r+1}^\infty (\sigma_k^{10})^2+
        2L_1(r)^2L_2(r)^2r^2\|u-\id_1\otimes Q_r\|_{(0,1)}^2\\
        &=
        \sum_{k=r+1}^\infty (\sigma_k^{10})^2+
        2L(r)^2r^2\sum_{k=r+1}^\infty(\sigma_k^{01})^2.
    \end{align*}
    Analogously we can estimate the $\|\cdot\|_{(0,1)}$ error. This completes
    the proof.
\end{proof}

To conclude this section, we extend the preceding result to $d>2$.
Unfortunately, unlike in the case for higher-order $L^2$-SVD in part I,
the upper bound will depend exponentially on
$d$. When performing an $L^2$-SVD in $d$ dimensions, the corresponding one
dimensional projectors are $L_2$-optimal. Thus, when considering
the tensor product $\mc P_{\bs r}$ of the projectors w.r.t.\ the
$\|\cdot\|_{e_j}$-norm for any $1\leq j\leq d$, only one factor in
$\mc P_{\bs r}$ is sub-optimal.

On the other hand, when the corresponding projectors are $H^1$-optimal and
we consider the tensor product $\mc P_{\bs r}$ of the projectors, all but one
factor are sub-optimal, yielding a constant that scales with an exponent
of $d-1$. Of course, for $d=2$ this is not obvious.

Before we proceed we introduce some notations to formalize the statement.
In analogy to \eqref{eq:Br}, we define the finite dimensional subspaces
\begin{align*}
    B^j_r:=\linspan{\psi_k^1:1\leq k\leq r}\subset\Umin_j(u),\quad
    1\leq j\leq d,
\end{align*}
where $\psi_k^j$ are the $H^1$-singular functions in the $j$-th dimension
(left singular functions of
$u_j:L^2(\bigtimes_{k\neq j}\Omega_k)
\rightarrow H^1(\Omega_j)$).
In principle we can take different ranks $r$ in each dimension, which only
results in a more cumbersome notation for the bound. We consider the
$H^1$-projectors
$
    P_r^j:H^1(\Omega_j)\rightarrow B^j_r,
$
and the corresponding tensorized versions
$
   \mc P_{r}^j=P_r^j\otimes\left(\bigotimes_{i\neq j}\id_i\right).
$
We introduce the index sets
$
    I^j:=\left\{1,\ldots,d\right\}\setminus\{j\}$, $1\leq j\leq d,
$
and the following sequence of sets
\begin{align}\label{eq:sets}
    I_1^j&=\emptyset,\notag\\
    I^j\supset I_{i}^j&\supset I^j_{i-1},\quad
    \#I_{i}^j=\#I_{i-1}^j+1\quad i=2,\ldots,d.
\end{align}
Note that the sets in this sequence are not unique. Apart from the first and
the last sets, there are finitely many possible combinations for the
intermediate sets.

\begin{theorem}\label{thm:h1errorh1l2svd}
    Let $u\in H^1(\Omega)$ and the assumptions of
    Lemma \cref{lemma:prbound} hold. I.e., we assume $u$ is twice weakly
    differentiable in each dimension. As before, we introduce the regularity
    factors
    \begin{align*}
        C_j(r)=\sup_{v_j\in B^j_r}\frac{\|v_j\|_1}{\|v_j\|_0}
               \sup_{v_j\in B^j_r}\frac{\|v_j\|_2}{\|v_j\|_1}.
    \end{align*}
    Take any sequence of sets $\{I_i^j\}_{i,j=1}^d$ as in \eqref{eq:sets}.
    Then, the $H^1$-error of the HOSVD projection can be estimated as
    \begin{align}\label{eq:h1l2hosvd}
        \left\|u-(\prod_{j=1}^d\mc P_{r}^j)u\right\|_1
        \leq\left(
        \sum_{j=1}^d\sum_{k=r+1}^\infty(\sigma_k^j)^2
        \left[\sum_{i=1}^d(\sqrt{2}r)^{2(i-1)}
        \prod_{l\in I_i^j}C_l(r)^2\right]
        \right)^{1/2}.
    \end{align}
\end{theorem}
\begin{proof}
    The result can be obtained by ``peeling off'' projectors.
    Observe that similar to Theorem \cref{thm:h1l2bound} we can write
    \begin{align*}
        \left\|u-(\prod_{j=1}^d\mc P_{r}^j)u\right\|^2_{e_k}
        &=
        \left\|u-\mc P_{r}^ku\right\|^2_{e_k}+
        \left\|\mc P_{r}^k[u-(\prod_{j\neq k}^d\mc P_{r}^j)u]\right\|^2_{e_k}\\
        &=
        \sum_{m=r+1}^\infty(\sigma_m^k)^2+
        \left\|\mc P_{r}^k[u-(\prod_{j\neq k}^d\mc P_{r}^j)u]\right\|^2_{e_k},
    \end{align*}
    for some $1\leq k\leq d$.
    For the latter term we apply the same arguments as in Theorem
    \cref{thm:h1l2bound}
    and obtain
    \begin{align*}
        \left\|\mc P_{r}^k[u-(\prod_{j\neq k}^d\mc P_{r}^j)u]\right\|^2_{e_k}
        &\leq 2r^2C_k(r)^2
        \left\|u-(\prod_{j\neq k}^d\mc P_{r}^j)u\right\|^2_{e_i},
    \end{align*}
    for some $i\neq k$. Next, we repeat this for $i$. I.e., for some
    $l\not\in\{k, i\}$
    \begin{align*}
        \left\|u-(\prod_{j\neq k}^d\mc P_{r}^j)u\right\|^2_{e_i}&=
        \sum_{m=r+1}^\infty(\sigma_m^i)+
        \left\|\mc P_r^i[
        u-(\prod_{j\not\in\{k, i\}}^d\mc P_{r}^j)u]\right\|^2_{e_i}\\
        &\leq
        \sum_{m=r+1}^\infty(\sigma_m^i)+
        2r^2C_i(r)^2
        \left\|u-(\prod_{j\not\in\{k, i\}}^d\mc P_{r}^j)u\right\|^2_{e_l}.
    \end{align*}

    The arbitrary order of choosing
    $i, l,\ldots$
    until we are left with just one projector leads to the arbitrary
    sequence of sets $I_i^j$ in \eqref{eq:h1l2hosvd}.
    This completes the proof.
\end{proof}

%% file: Sections/alt.tex
\section{Alternative Forms of Low-Rank Approximation in $H^1$}\label{sec:alt}
In this section we investigate alternative approaches for
low-rank approximation
with error control in $H^1$.

\subsection{Spaces of Mixed Smoothness}
Consider again a function $u\in H^1(\Omega_1)\otimes_a H^1(\Omega_2)$ viewed
as an operator
$
    u:H^1(\Omega_2)\rightarrow H^1(\Omega_1).
$
Completing $H^1(\Omega_1)\otimes_a H^1(\Omega_2)$ w.r.t.\ the canonical norm
$\|\cdot\|_{\mix}$ leads to $H^1_{\mix}(\Omega)$.
For $d=2$ we have the inclusions
$
    H^2(\Omega)\subset H^1_{\mix}(\Omega)\subset H^1(\Omega).
$
Thus, assuming additionally $u\in H^1_{\mix}(\Omega)$ is not a severe
regularity restriction. In particular solutions to elliptic PDEs will often
satisfy this assumption. However, for general $d\geq 2$, we have the inclusions
$
    H^d(\Omega)\subset H^1_{\mix}(\Omega)\subset H^1(\Omega).
$
As the dimension grows, the regularity restriction becomes more and more
severe.
Nonetheless, there are important examples where such assumptions are valid,
e.g., for
the solution to the Schr\"{o}dinger equation, see \cite[Chapter 6]{HY}.

One can ask if we can exploit the SVD w.r.t.\ the $\|\cdot\|_{\mix}$-norm
in higher dimensions without assuming dimension dependent regularity. To
this end, for general $d\geq 2$, we consider $u\in H^1(\Omega)$ such
that all mixed derivatives of order 2 exist, i.e.,
$
    \frac{\partial^2}{\partial x_i\partial x_j} u,$ $1\leq i,j\leq d,\;
    i\neq j,
$
exist in the weak sense and are $L^2$-integrable. Define the spaces
$
    \mb V_j:=H^1(\Omega_j)\otimes_a H^1(\bigtimes_{i\neq j}\Omega_i),
$
with the corresponding norm
\begin{align*}
    \|u\|^2_{\mix, j}:=\|u\|_0^2+\sum_{i=1}^d\left\|\frac{\partial}
    {\partial x_i}u\right\|^2_0+\sum_{i\neq j}\left\|\frac{\partial^2}
    {\partial x_j\partial x_i}u\right\|^2_0.
\end{align*}
A new intersection space is defined via
$
    \mb V:=\bigcap_{j=1}^d\mb V_j,$ $
    \|\cdot\|^2_{\mb V}:=\sum_{j=1}^d\|\cdot\|^2_{\mix, j}.
$
In each $\mb V_j$
there exists an optimal rank $r$ approximation w.r.t.\ the
$\|\cdot\|_{\mix, j}$-norm that we call $u^j_r$.
We can define the corresponding minimal subspaces as
$
    M^j_r:=\Umin_{j}(u^j_r)\subset H^1(\Omega_j),$ $\dim{M^j_r}=r.
$
The $H^1$-orthogonal projection is denoted by
$
    P^j_r:H^1(\Omega_j)\rightarrow M^j_r.
$
We consider the HOSVD projection
$
    \mc P_{r}:=\bigotimes_{j=1}^d P_r^j.
$
As before, for simplicity we take $r$ constant and independent of $j$, but in
principle the extension to different $r_j$ is straightforward.
Before we proceed, we briefly justify why such a projection makes sense on
$\mb V$.
\begin{lemma}
    Let $A_j:X_j\rightarrow Y_j$ be linear and continuous operators between
    Hilbert spaces $X_j$ and $Y_j$, $1\leq j\leq d$. Define
$
        \mb X:=\clos{\|\cdot\|_{\bs X}}{{}_a\bigotimes_{j=1}^d X_j},$ $
        \mb Y:=\clos{\|\cdot\|_{\bs Y}}{{}_a\bigotimes_{j=1}^d Y_j},
$
    where $\|\cdot\|_{\mb X}$ and $\|\cdot\|_{\mb Y}$ are the canonical norms
    induced by the Hilbert spaces $X_j$ and $Y_j$. Then, the operator
$
        A:=\bigotimes_{j=1}^d A_j:\mb X\rightarrow\mb Y,
$
    is well defined, i.e., can be uniquely extended to a continuous operator
    on $\mb X$. For the operator norm we get
$
        \|A\|=\prod_{j=1}^d \|A_j\|.
$
\end{lemma}
\begin{proof}
    One can follow the same arguments as in \cite[Proposition 4.127]{HB}.
\end{proof}

Since $P^j_r:H^1(\Omega_j)\rightarrow H^1(\Omega_j)$ is bounded and by applying
the preceding lemma, we note that
\begin{align}\label{eq:projmix}
    \mc P_r^j:=P_r^j\otimes\left(\bigotimes_{i\neq j}\id_i
    \right):\mb V_j\rightarrow \mb V_i,
\end{align}
is bounded for any $1\leq i\leq d$ and any $1\leq j\leq d$. Thus, 
the projections from \eqref{eq:projmix} are well defined on $\mb V$, commute and
the composition $\mc P_{r}$ is well defined as well.

We are now ready to derive an error estimate for the HOSVD projection.
Unfortunately, we can only slightly improve the bound in
\eqref{eq:h1l2hosvd}, as the next statement shows. Once again, we will require
the projections above to be bounded in $L^2$. This will lead to a higher
regularity requirement $u\in H^3(\Omega)$.

\begin{restatable}{proposition}{hmix}\label{prop:h1mixgen}
    Let $d>2$, $u\in H^3(\Omega)$ and $u$ satisfy Dirichlet or
    Neumann boundary conditions as in \cref{lemma:prbound}.
    As before, we define the regularity
    factors
    \begin{align*}
        D_j(r):=\sup_{v_j\in M^j_r}\frac{\|v_j\|_2}{\|v_j\|_1}.
    \end{align*}

    Let $I=(1,\ldots, d)$ be an ordered tuple with the indexing convention
    $I(j)=j$.
    Denote by $S_d(I)$
    the set of all possible permutations of $I$\footnote{We use a slight
    abuse of notation for the permutation group.}.

    Then, with the shorthand notation
    $j^c:=\{1,\ldots,d\}\setminus\{j\}$ for any $1\leq j\leq d$, we can
    estimate the HOSVD projection error as
    \begin{align*}
        &\|u-\mc P_{r}u\|_1\\
        &\leq
        2r^{\frac{d-2}{2}}
        \min_{J\in S_d(I)}
        \sum_{j=1}^d
        \left[
        \max_{i\in (J(j))^c}\prod_{\substack{k=J(1),\ldots,J(j-1),\\k\neq i}}
        D_k(r)
        \right]\cdot
        \left[\sum_{k=r+1}^\infty(\sigma_k^{J(j)})^2\right]^{1/2}
    \end{align*}
    where $\{\sigma_k^j\}_{k\in\N}$, $1\leq j\leq d$, are the HOSVD singular
    values.
\end{restatable}
\begin{proof}
    See Appendix \cref{sec:app}.
\end{proof}

The above bound is similar in nature to \cref{thm:h1errorh1l2svd}.
In both cases the exponential
dependance on $d$ arises since $d-1$ $H^1$-orthogonal projections involved in
$\mc P_{r}$ are sub-optimal.

We conclude this subsection by providing bounds for the $H^1$-error using
$H^1_{\mix}$-singular values.
We derive
the result for $d=2$. Unlike in for higher-order $L^2$-SVD in part I,
this result does not possess
an elegant generalization to $d>2$ for the same reason the statements above
introduce factors depending exponentially on the dimension.

Let $d=2$ and $\{\sigma_k^{11}\}_{k\in\N}$ denote the singular values
associated with
the $H^1_{\mix}$-SVD. Let $\{\psi^\mix_k\}_{k\in\N}$ and
$\{\phi_k^\mix\}_{k\in\N}$
denote the corresponding left and right singular functions. Then, the best
rank $r$ approximation w.r.t.\ $\|\cdot\|_{\mix}$ is given by
$
    u_r=\sum_{k=1}^r\sigma_k^{11}\psi_k^\mix\otimes\phi_k^\mix.
$

\begin{proposition}
    For $u\in H^1_{\mix}(\Omega)$ we have the following upper
    and lower\\bounds for the $H^1$ error
    \begin{align*}
        \|u-u_r\|_1^2&\leq\sum_{k=r+1}^\infty
        (\sigma_k^{11})^2(\|\phi_k^1\|_0^2+\|\psi_k^1\|_0^2),\\
        \|u-u_r\|_1^2&\geq\frac{1}{2}\sum_{k=r+1}^\infty
        (\sigma_k^{11})^2(\|\phi_k^1\|_0^2+\|\psi_k^1\|_0^2).
    \end{align*}
\end{proposition}
\begin{proof}
    The proof follows the same lines as the one of \cite[Theorem 4.1]{ANSVD}.
\end{proof}

\subsection{Exponential Sums}
One can reformulate the problem of low-rank approximations in $H^1$ as a
problem on sequence spaces. This point of view is particularly close to
numerical application and, in essence, has already been applied in previous
works, as we will demonstrate below.
For ease of exposition we will consider Fourier bases. But in principle any
multiscale Riesz basis could be used, e.g., wavelets.

Let $u\in H^1([-\pi,\pi]^2)$ be a $2\pi$-periodic function.
Then, $u$ can be expanded in the Fourier basis as
$
    u(x,y)=\frac{1}{2\pi}\sum_{k,m\in\Z}c_{km}e^{ikx}e^{imy},
$
where we also know that
$
    \sum_{k,m\in\Z}|c_{km}|^2(1+k^2+m^2)<\infty.
$

Since the Fourier basis is orthonormal in $L^2$, performing an SVD of the
sequence $\{c_{km}\}_{k,m\in\Z}$, we implicitly
obtain an $L^2$-SVD of $u$.
Since the Fourier basis is orthogonal in $H^1$ as well, we can simply rescale
and perform an SVD on the resulting sequence. However, this time with
error control in $H^1$.

More precisely,
\begin{align*}
    u(x,y)=\frac{1}{2\pi}\sum_{k,m\in\Z}c_{km}e^{ikx}e^{imy}
    =\frac{1}{2\pi}\sum_{k,m\in\Z}c_{km}
    \frac{\sqrt{1+k^2+m^2}}{\sqrt{1+k^2+m^2}}e^{ikx}e^{imy}.
\end{align*}
Performing the $\ell_2$-SVD of $\{c_{km}\sqrt{1+k^2+m^2}\}_{k,m\in\Z}$,\\
$
    c_{km}\sqrt{1+k^2+m^2}=\sum_{l=1}^\infty\sigma_lv_k^lw_m^l,
$
we obtain
\begin{align*}
    u(x,y)&=\frac{1}{2\pi}\sum_{l=1}^\infty\sigma_l
    \sum_{k,m\in\Z}
    \frac{1}{\sqrt{1+k^2+m^2}}v_k^l e^{ikx} w_m^l e^{imy}.
\end{align*}
The remaining issue is that the functions
$
    \sum_{k,m\in\Z}\frac{1}{\sqrt{1+k^2+m^2}}v_k^l e^{ikx} w_m^l e^{imy},
$
are not separable due to the scaling term $\frac{1}{\sqrt{1+k^2+m^2}}$.
On the other hand, the latter can be approximated to any desired accuracy by
exponential sums (see also \cite[Chapter 9.7.2]{HB}), which in turn are
separable. We approximate in the form
\begin{align}\label{eq:expsum}
    \frac{1}{\sqrt{1+k^2+m^2}}\approx\sum_{\nu\in\Z}E^\delta(k,\nu)E^\delta
    (m,\nu),
\end{align}
where $\delta>0$ controls the accuracy of the approximation.
Finally, we get the separable representation
\begin{align*}
    u(x, y)\approx
    \frac{1}{2\pi}\sum_{l=1}^\infty\sigma_l
    \sum_{\nu\in\Z}
    \left(
    \sum_{k\in\Z}E^\delta(k,\nu)
    v_k^l e^{ikx}\right)
    \left(\sum_{m\in\Z}
    E^\delta(m,\nu)w_m^le^{imy}\right),
\end{align*}
where the approximation can be performed to any accuracy
$\delta>0$. A finite
representation involves truncating the Fourier basis representation w.r.t.\
$k$ and $m$, truncating the exponential sum approximation w.r.t.\ $\nu$,
and truncating to a low-rank representation w.r.t.\ $l$. If we denote the
number of Fourier basis terms in each dimension by $n$, the number of
exponential sum terms by $p$ and the rank bound by $r$, then the overall
complexity for such a representation is $\mc O(rn2p)$, with a rank of
the final representation bounded by $rn$.

In principle, the same type of SVD was applied in \cite{BDH1}. There the authors
constructed an adaptive wavelet solver based on inexact Richardson iterations
for elliptic equations. They introduced a separable exponential sum
preconditioner, which approximates the scaling coefficients similar to
\eqref{eq:expsum}.
The properly scaled coefficients of the
numerical solution were then truncated via HOSVD. This is implicitly equivalent
to the procedure above.

A similar approach
was performed in \cite{BDL2} and \cite{AU}. In \cite{BDL2} the authors
controlled the error only in $L^2$ but generally observed convergence in
$H^1$ as well. This is consistent with our analysis for the $L^2$-SVD in part I.

\subsection{Sobolev Functions as Operators}
Until now we considered low-rank approximations for $u\in H^1(\Omega)$ by
using the $L^2$-SVD, $H^{(1,0)}$-SVD, $H^{(0,1)}$-SVD
and $H^1_{\mix}$-SVD. In all cases we
required additional regularity assumptions and the error estimates involved
singular values and scaling factors. One could ask if there is a natural
interpretation of $u\in H^1(\Omega)$ that fully exploits the intersection space
structure without any additional assumption.

For simplicity we consider the case $d=2$ and the space
$
    H^1(\Omega)\cong
    H^{(1,0)}\cap H^{(0,1)},
$
where on the right hand side we use the intersection norm
$
    \|\cdot\|_{\cap}^2:=\|\cdot\|^2_{(1,0)}
    +\|\cdot\|^2_{(0,1)}.
$
The structure of the norm suggests it is more appropriate to consider a direct
sum space. Thus, we define
$
    H_{2D}:=H^{(1,0)}\times H^{(0,1)},
$
with the corresponding natural norm
$
    \|\cdot\|_{2D}^2:=\|\cdot\|^2_{(1,0)}
    +\|\cdot\|^2_{(0,1)}.
$
We can continuously embed $H^1(\Omega)$ into this space via the linear
isometry
$
    H^1(\Omega)\hookrightarrow H_{2D},$ $
    u\mapsto (u, u),$ $\|u\|_1\sim \|(u, u)\|_{2D}.
$
The space $H^1(\Omega)$ represents the ``diagonal'' of $H_{2D}$. To see how $u$
can represent an operator, we further embed $H_{2D}$ into a space of Hilbert
Schmidt operators
\begin{align*}
    \HS\left(L^2(\Omega_2)\times H^1(\Omega_2),
    H^1(\Omega_1)\times L^2(\Omega_1)\right),
\end{align*}
by identifying $(u_1,u_2)\in H_{2D}$ with a map
$
   (u_1, u_2)[(w, v)]:=(u_1[w], u_2[v]),$ $
    \|(u_1, u_2)\|_{2D}=\|(u_1, u_2)\|_{\HS}.
$
To see the norm identity, consider again the $H^{(1,0)}$- and $H^{(0,1)}$-SVDs,
$
    u_1=\sum_{k=1}^\infty \sigma_k^{10}\psi_k^1\otimes\phi_k^0,$ $
    u_2=\sum_{k=1}^\infty \sigma_k^{01}\psi_k^0\otimes\phi_k^1.
$
Since $\{\phi^0_k\}_{k\in\N}$ and
$\{\phi_k^1\}_{k\in\N}$ are complete orthonormal systems for $L^2(\Omega_2)$
and\\
$H^1(\Omega_2)$, respectively, $\{(\phi_k^0, 0), (0, \phi_k^1)\}_{k\in\N}$
is a complete orthonormal system for\\
$L^2(\Omega_2)\times H^1(\Omega_2)$.
Analogously for $\{(\psi_k^1, 0), (0, \psi_k^0)\}_{k\in\N}$.
Applying $(u_1, u_2)$ to this orthonormal system we get
\begin{align*}
    (u_1, u_2)[(\phi_i^0, 0)]&=\left(\sum_{k=1}^\infty\sigma_k^{10}
    \dpair{\phi_i^0}{\phi_k^0}_0\psi_k^1, 0\right)=\sigma_i^{10}(\psi_i^1, 0),\\
    (u_1, u_2)[(0, \phi_i^1)]&=\left(0, \sum_{k=1}^\infty\sigma_k^{01}
    \dpair{\phi_i^1}{\phi_k^1}_1\psi_k^0\right)=\sigma_i^{01}(0, \psi_i^0).
\end{align*}
Let $\{\sigma_k^\cup\}_{k\in\N}$ represent the sorted union of the singular
values $\{\sigma_k^{10}\}_{k\in\N}$ and $\{\sigma_k^{01}\}_{k\in\N}$. Then, by
the above, the SVD of $(u_1, u_2)$ is given by
$
    (u_1, u_2) = \sum_{k=1}^\infty\sigma_k^\cup\psi_k\otimes\phi_k,
$
where
\begin{align*}
    \psi_k\otimes\phi_k=
    \begin{cases}
        (\psi_l^1, 0)\otimes (\phi_l^0, 0),&\;\text{if}\quad\sigma_k=\sigma_l^{10},\\
        (0, \psi_l^0)\otimes (0, \phi_l^1),&\;\text{if}\quad\sigma_k=\sigma_l^{01}.
    \end{cases}
\end{align*}

The extension to $d>2$ is straightforward. In summary, $H_{2D}$ seems like
a natural space for low-rank approximations for $u\in H^1(\Omega)$ and in
which $u$ can be interpreted as a Hilbert Schmidt operator without any
additional assumptions.

The issue remains, however, that low-rank approximations in $H_{2D}$ involve
a pair of approximations: one for the left and one for the right derivative.
If we require a single low-rank approximation, we would have to project
onto the ``diagonal'' of $H_{2D}$. This essentially involves the application
of the inverse of the Laplacian, which is not separable.

A general approach might be to reformulate a problem given in $H^1(\Omega)$ into
a problem in $H_{2D}$ and solve the latter in a low-rank format to obtain
a solution
being a tuple of low-rank approximantions. In a last step, one could apply an
approximate, efficient and problem independent projection onto the diagonal
to obtain a low-rank approximation $u_r\in H^1(\Omega)$. Though natural,
it is unclear to us if and how the interpretation as $u\in H_{2D}$ is
of practical use.

%% file: Sections/numexp.tex
\section{Numerical Experiments}\label{sec:numexp}
In this section we verify our findings with a few toy examples. First,
we consider a function $u\in H^1([-\pi, \pi]^2)$, expand this function
in Fourier bases and truncate the expansion
\begin{align*}
    u(x, y)=\frac{1}{2\pi}\sum_{k,m\in\Z}c_{km}e^{-ikx}e^{-imy}
    \approx u_n(x, y)=
    \frac{1}{2\pi}\sum_{k=-n}^n\sum_{m=-n}^nc_{km}e^{-ikx}e^{-imy}.
\end{align*}
Then, we perform an SVD of $u_n$. This situation is prototypical for a
numerical method, where the current numerical approximation $u_n$ (with
possibly high ranks) is truncated
to a low-rank approximation $\tilde{u}_n$. We are particularly interested
in the behavior of the singular values and comparisons with $L^2$ and
$H^1$ errors.

We consider two functions. First,
$
    u(x, y) = (x^2+y^2)^{0.3},
$
which has a singularity in the derivatives at
$x=y=0$. Second,
$
    u(x, y) = |x+y|^{0.6},
$
which has a singularity along the anti-diagonal $x=-y$. The results
are displayed in \cref{fig:svd}.

The singular
values of the first function decay faster.
For the second function, since the singularity is not axis aligned,
we expect bad separability.
We plot both the $L^2$ and $H^1$ errors of the
$L^2$-SVD. We also plot the $H^1$-error of the projection
$(P_r\otimes Q_r) u$ from
\eqref{eq:comm}. In both cases $(P_r\otimes Q_r)u$ 
does not improve the error of the $L^2$-SVD.

Moreover, we also compare this with the best possible approximation in the
following sense. We take the eigenfunctions generated by all SVDs:
$L^2$-eigenfunctions of the $L^2$-SVD, $H^1$-eigenfunctions of the
$H^{(1,0)}$-, $H^{(0,1)}$-SVDs and
$L^2$-eigenfunctions of the $H^{(1,0)}$-, $H^{(0,1)}$-SVDs. Then, we perform
an $H^1$-orthogonal projection onto the space of tensor products spanned by
all possible combinations of these eigenfunctions. Of course, such a procedure
is not feasible in higher dimensions, it serves merely to illustrate our point.
We denote this by ``$H^1$ error optimal approximation''.

As can be seen in the plot for the second function, all possible
projections are the same as the best possible one. This is consistent with
expectation. In fact, all of the eigenspaces mentioned above are the same,
i.e., the eigenfunctions are linearly dependent. Recall the definition of
the three possible eigenspaces:
\begin{align*}
    U_{I}^j(u_n) &:= \clos{\|\cdot\|_0}
    {\linspan{\varphi(u_n):\varphi=\bigotimes_{k=1}^2\varphi_k,\;
    \varphi_j=\id_j,\;\varphi_k\in\left(L_2(\Omega_k))^*,\;k\neq j\right)}},\\
    U_{II}^j(u_n) &:= \clos{\|\cdot\|_1}
    {\linspan{\varphi(u_n):\varphi=\bigotimes_{k=1}^2\varphi_k,\;
    \varphi_j=\id_j,\;\varphi_k\in\left(L_2(\Omega_k))^*,\;k\neq j\right)}},\\
    U_{III}^j(u_n) &:= \clos{\|\cdot\|_0}
    {\linspan{\varphi(u_n):\varphi=\bigotimes_{k=1}^2\varphi_k,\;
    \varphi_j=\id_j,\;\varphi_k\in\left(H^1(\Omega_k))^*,\;k\neq j\right)}}.
\end{align*}
Since $u_n\in H^1(\Omega_1)\otimes_a H^1(\Omega_2)$,
by \cite[Lemma 3.1]{ANSVD} (see also \cite[Remark 6.32]{HB}),
$U_I^j(u_n)=U_{II}^j(u_n)=U_{III}^j(u_n)$.
From a theoretical perspective,
the truly difficult cases are when $u\in H^1(\Omega)$ but is not in
$H^1(\Omega_1)\otimes_a H^1(\Omega_2)$.
Only in such cases
the minimal subspaces depend on the topology of the ambient space.
In particular, this means that if
$u\in H^1(\Omega)$ is a numerical approximation, most of the assumptions
in the previous section hold\footnote{With sufficient regularity of the
basis functions, all assumptions hold.}.

We use an error estimator for the $H^1$-error
\begin{align}\label{eq:ee}
    e(r)=\left(\sum_{k=r+1}^n(\sigma^{10}_k)^2+(\sigma_k^{01})^2
    \right)^{1/2},
\end{align}
where $\{\sigma_k^{10}\}_{k\in\N}$ and $\{\sigma_k^{01}\}_{k\in\N}$ are the
singular values from Proposition \cref{prop:relationsingval}. The projections
$P_r$, $Q_r$ are from Section \cref{sec:h1l2proj}.
As can be seen in both plots, this error estimator lies perfectly on the
$H^1$-error. This is consistent with \cite[Theorem 4.1]{ANSVD} and Theorem
\cref{thm:singvalsbounds}.

\begin{figure}[ht!]
    \begin{center}
        \includegraphics[width=0.65\linewidth]{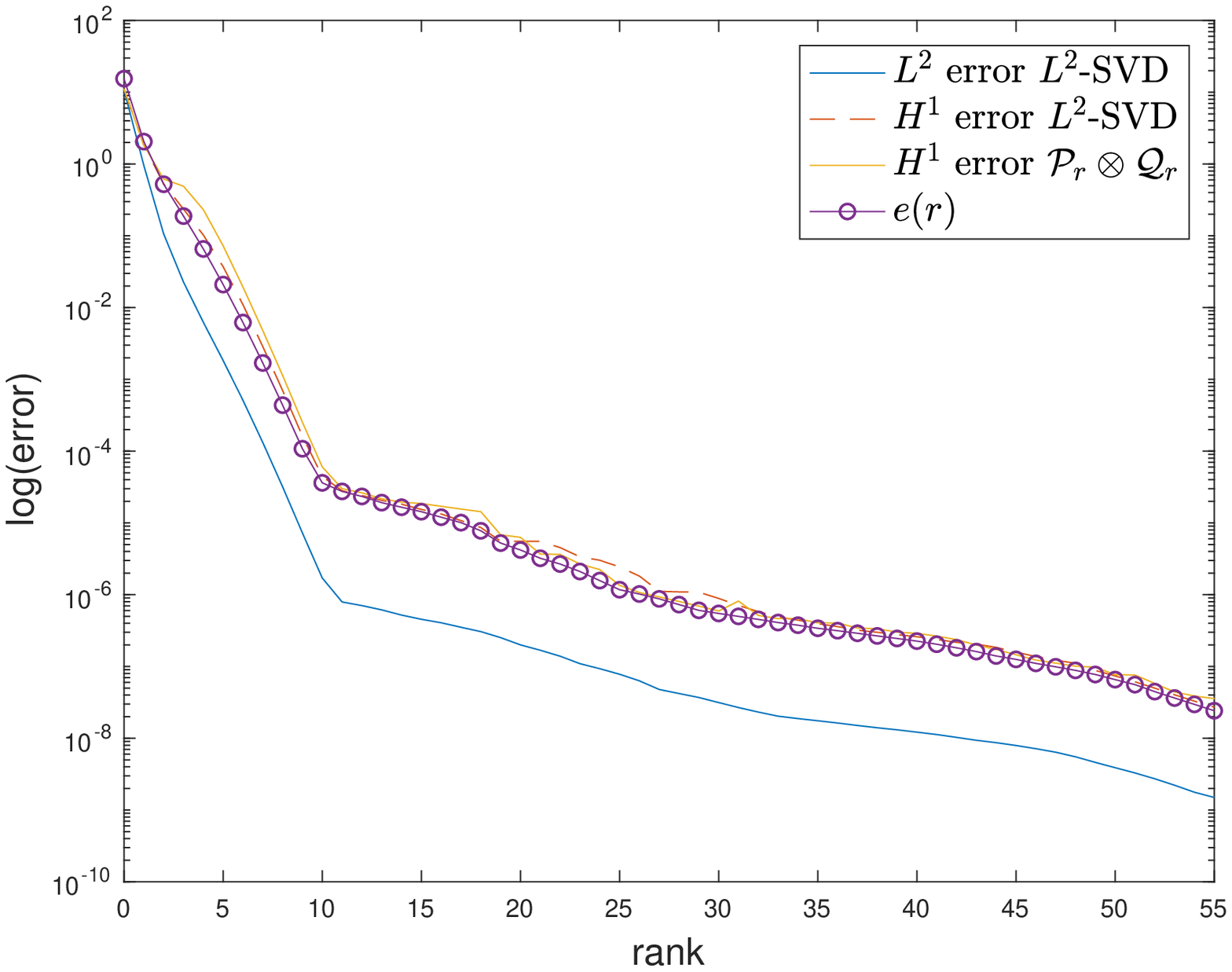}
        \includegraphics[width=0.65\linewidth]{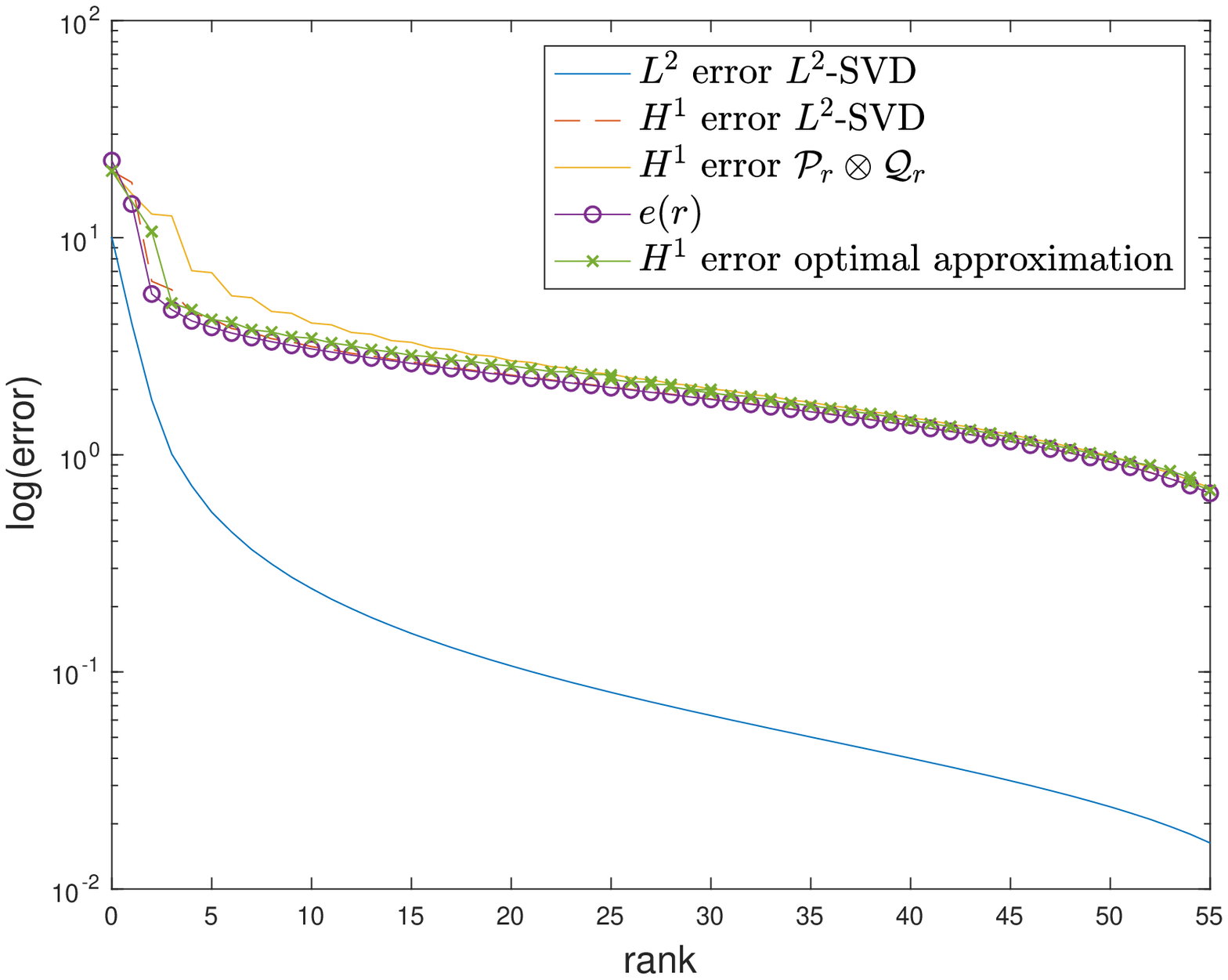}
        \captionof{figure}{Low rank approximations for truncated Fourier
                           series of $u(x, y)=(x^2+y^2)^{0.3}$ (top),
                           $u(x, y)=|x+y|^{0.6}$ (bottom).}
    \label{fig:svd}
    \end{center}
\end{figure}

These findings suggest that we can compute a low-rank approximation for
$u_n$ by performing an $L^2$-SVD and truncating based on the error estimator
in \eqref{eq:ee} to control the error in $H^1$. In the following we do just that.
We consider the weak formulation of the Poisson equation
$
    -\Delta u= f.
$
We compute a Galerkin approximation $u_n\approx u$, and truncate this
approximation to $\tilde{u}_n$ such that
\begin{align*}
    \|u-u_n\|_1\leq \|u-\tilde{u}_n\|_1\leq 2 \|u-u_n\|_1.
\end{align*}
We increase the discretization size $n$, i.e., the number $n^2$ of basis functions.
The results are displayed in \cref{fig:pde}. The plotted errors are
approximations to the exact errors $\|u-u_n\|_1$ and
$\|u-\tilde{u}_n\|_1$.
In both cases the error bounds
are fulfilled and the rank of $\tilde{u}_n$ remains below $5$.

\begin{figure}[ht!]
    \begin{center}
        \includegraphics[width=0.65\linewidth]{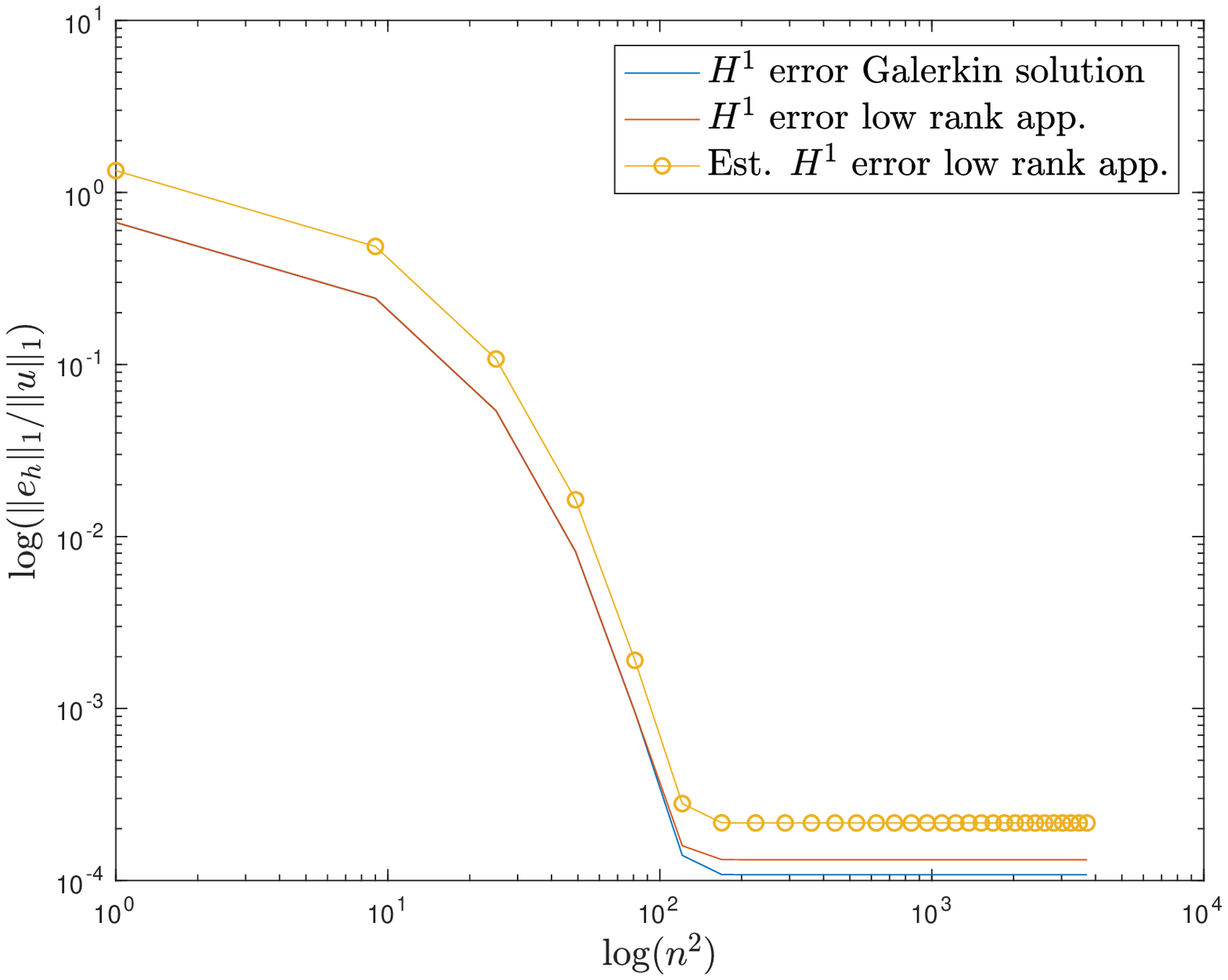}
        \includegraphics[width=0.65\linewidth]{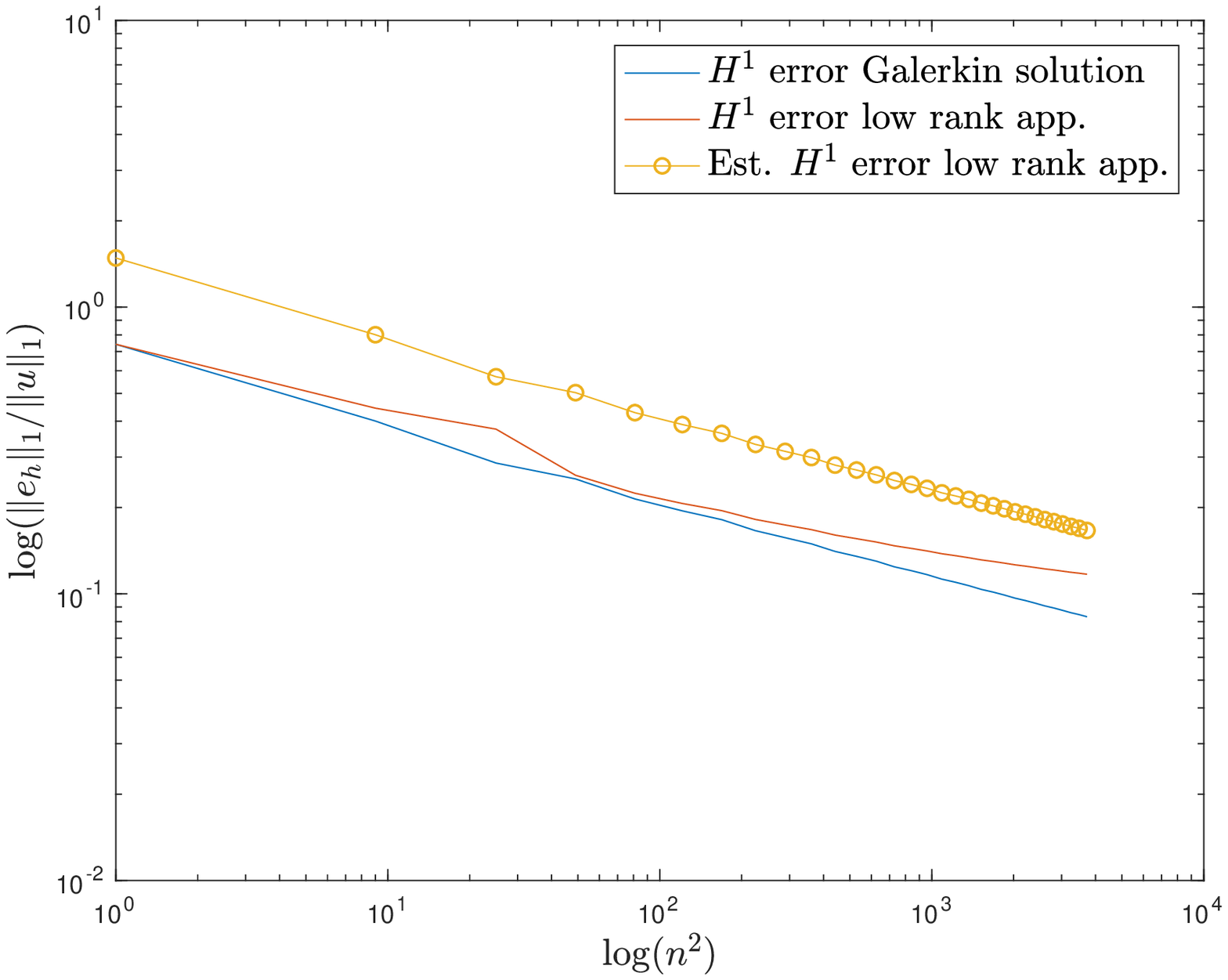}
        \captionof{figure}{Galerkin solutions for reference functions
                           $u(x, y)=\exp(\cos(x)\cos(y))$ (top)
                           and $u(x, y)=|1-(x^2+y^2)|^{0.95}$ (bottom).}
    \label{fig:pde}
    \end{center}
\end{figure}

%% file: Sections/conclusion.tex
\section{Conclusion}\label{sec:concl}
We proposed and analyzed several variants of
low-rank approximations of functions in Sobolev spaces.
In part I, we show that
sets of functions with bounded Tucker (multi-linear) rank
in Sobolev spaces are weakly closed. Sobolev functions can
be shown to be in the tensor product of their minimal subspaces under certain
conditions, such as additional regularity. However, we do not believe that
this holds in general.
The $L^2$-SVD preserves regularity of the decomposed functions and, under
certain conditions, we can quantify the $H^1$ error in terms of the rescaled
singular values.

In part II,
we show that the singular values of different SVDs are closely
related. Lower and upper bounds are obtained by simple scalings.
We also analyze $H^1$ minimal subspaces. The SVD in $H^{(1,0)}$ does not
preserve regularity and $H^1$ bounds require additional smoothness.
The resulting bounds are worse than that of the $L^2$-SVD. Similar
bounds apply to spaces of lower order mixed smoothness for $d>2$.
This indicates the $L^2$-SVD performs better for low-rank
approximations than variants of SVDs involving Sobolev spaces.

Numerical experiments are consistent with the analytical findings.
Differences
between minimal subspaces w.r.t.\ to different norms arise only when
considering functions in Sobolev spaces that are not in the
algebraic tensor spaces.
For
constructing low-rank approximations of numerical solutions, the
different types of minimal subspaces do not add information. However,
the singular values of $H^{(1,0)}$- and $H^{(0,1)}$-SVDs
are better suited to estimate the
$H^1$ error and, for numerical purposes, it seems the best recipe are
low-rank approximations built from $L^2$-SVDs but with $H^{(1,0)}$ and
$H^{(0,1)}$ singular
values used for $H^1$-error control.

Finally, we briefly mentioned alternatives. Exponential sums are a well known
technique already utilized in previous works. On the other hand, if one
pursues the viewpoint of Sobolev spaces being intersection spaces, a natural
approach would be to consider direct sum spaces. We briefly introduced this
viewpoint.

There are a few immediate open questions that arise in conclusion of this work.
It would be interesting to consider how the above analysis extends
to hierarchical tensor formats (see \cite[Chapter 11]{HB}).
Numerical experiments for high-dimensional
problems with a fine or adaptive discretization
should shed more light on the performance of SVD in Sobolev spaces.

%% file: Sections/ack.tex
\section*{Acknowledgments}
We would like to thank Jochen Gl\"{u}ck\footnote{jochen.glueck@alumni.uni-ulm.de}
for Example \cref{ex:discproj}.

%% file: Sections/appendix.tex
\appendix
\section{Proof of \cref{prop:h1mixgen}}\label{sec:app}
\begin{proof}
    \textbf{I.} To begin, we consider the statement for $J=I=(1,\ldots, d)$.
    We have
    \begin{align*}
        &\|u-\mc P_{r}u\|_1=
        \left\|\left(\id-\prod_{k=1}^d P_r^k\bigotimes_{i\neq k}
        \id_i\right)u\right\|_1
        =\left\|\sum_{j=1}^d\prod_{k=1}^{j-1}
        P_r^k\bigotimes_{i\neq k}\id_i\left(u-\mc P_r^ju\right)\right\|_1\\
        &\leq\sum_{j=1}^d
        \left\|\prod_{k=1}^{j-1}
        P_r^k\bigotimes_{i\neq k}\id_i\left(u-\mc P_r^ju\right)\right\|_1
        \leq
        \sum_{j=1}^d
        \left\|\prod_{k=1}^{j-1}
        P_r^k\bigotimes_{i\neq k}\id_i\left(u-\mc P_r^ju\right)\right\|_{\mix, j}\\
        &=
        \sum_{j=1}^d
        \left\|\left(P_r^1\otimes\cdots\otimes P_r^{j-1}\otimes\id_j
        \otimes\cdots\otimes\id_d\right)\left(u-\mc P_r^ju\right)\right\|_{\mix, j}
    \end{align*}
    Thus, we need to bound the norm of the operator
    \begin{align*}
        P_r^1\otimes\cdots\otimes P_r^{j-1}\otimes\id_j
        \otimes\cdots\otimes\id_d:
        \clos{\|\cdot\|_{\mix, j}}{\mb V_j}\rightarrow
        \clos{\|\cdot\|_{\mix, j}}{\mb V_j}.
    \end{align*}
    Since $\|\cdot\|_{\mix, j}$ is a uniform crossnorm on 
    $\clos{\|\cdot\|_{\mix, j}}{\mb V_j}$, we only need to bound
    \begin{align*}
        P_r^1\otimes\cdots\otimes P_r^{j-1}\otimes
        \id_{j+1}\otimes\cdots\otimes\id_d:
        H^1(\bigtimes_{k\neq j}\Omega_k)\rightarrow
        H^1(\bigtimes_{k\neq j}\Omega_k).
    \end{align*}

    \textbf{II.} To that end, we first check if $P_r^j$ is bounded in
    $L^2$. In order to describe the space $M^j_r$, we consider again the SVD of
    the operator $u:H^1(\Omega_j)\rightarrow H^1(\bigtimes_{k\neq j}\Omega_k)$.
    To shorten notation, we use
    $x_j^\wedge:=(x_1,\ldots,x_{j-1}, x_{j+1},\ldots,x_d)$ and
    $x_j^\vee:=(\ldots,x_j,\ldots)$. We have
    \begin{align*}
        u[v]=\int_{\Omega_j}u(x_j^\vee)
        v(x_j)dx_j+\int_{\Omega_j}\frac{\partial}{\partial x_j}
        u(x_j^\vee)
        \frac{d}{dx_j}v(x_j)dx_j,
    \end{align*}
    and
    \begin{align*}
        &u^*[w]=\int_{\bigtimes_{k\neq j}\Omega_k}u(\ldots, x_{j-1},\cdot,
        x_{j+1},
        \ldots)
        w(x_j^\wedge)dx_j^\wedge\\
        &+
        \sum_{i\neq j}
        \int_{\bigtimes_{k\neq j}\Omega_k}\frac{\partial}{\partial x_i}
        u(\ldots, x_{j-1},\cdot,
        x_{j+1},
        \ldots)
        \frac{\partial}{\partial x_i}v(x_j^\wedge)dx_j^\wedge.
    \end{align*}
    The singular functions $\psi_k^j\in H^1(\Omega_j)$ satisfy
    \begin{align*}
        &u^*u[\psi_k]=\int_{\bigtimes_{k\neq j}\Omega_k}
        u(\ldots, x_{j-1},\cdot,x_{j+1},\ldots)
        \int_{\Omega_j}u(x)
        \psi_k^j(x_j)dx_jdx_j^\wedge\\
        &+
        \int_{\bigtimes_{k\neq j}\Omega_k}
        u(\ldots, x_{j-1},\cdot,x_{j+1},\ldots)
        \int_{\Omega_j}\frac{\partial}{\partial x_j}u(x)
        \frac{d}{dx_j}\psi_k^j(x_j)dx_jdx_j^\wedge\\
        &+
        \sum_{i\neq j}
        \int_{\bigtimes_{k\neq j}\Omega_k}\frac{\partial}{\partial x_i}
        u(\ldots, x_{j-1},\cdot,
        x_{j+1},
        \ldots)
        \int_{\Omega_j}
        \frac{\partial}{\partial x_i}
        u(x)\psi_k^j(x_j)dx_j
        dx_j^\wedge\\
        &+
        \sum_{i\neq j}
        \int_{\bigtimes_{k\neq j}\Omega_k}\frac{\partial}{\partial x_i}
        u(\ldots, x_{j-1},\cdot,
        x_{j+1},
        \ldots)
        \int_{\Omega_j}
        \frac{\partial^2}{\partial x_i\partial x_j}
        u(x_j^\vee)
        \frac{d}{dx_j}\psi_k^j(x_j)dx_j
        dx_j^\wedge\\
        &=\lambda_k^j\psi_k^j,
    \end{align*}
    where $\lambda_k^j=(\sigma_k^j)^2$. Since $u\in H^3(\Omega)$,
    differentiating twice we get
    \begin{align*}
        &\lambda^j_k\frac{d^2}{dx_j^2}\psi_k=\int_{\bigtimes_{k\neq j}\Omega_k}
        \frac{\partial^2}{\partial x_j^2}u(\ldots, x_{j-1},\cdot,x_{j+1},\ldots)
        \int_{\Omega_j}u(x)
        \psi_k^j(x_j)dx_jdx_j^\wedge\\
        &+
        \int_{\bigtimes_{k\neq j}\Omega_k}
        \frac{\partial^2}{\partial x_j^2}u(\ldots, x_{j-1},\cdot,x_{j+1},\ldots)
        \int_{\Omega_j}\frac{\partial}{\partial x_j}u(x)
        \frac{d}{dx_j}\psi_k^j(x_j)dx_jdx_j^\wedge\\
        &+
        \sum_{i\neq j}
        \int_{\bigtimes_{k\neq j}\Omega_k}\frac{\partial^3}{\partial x_j^2
        \partial x_i}
        u(\ldots, x_{j-1},\cdot,
        x_{j+1},
        \ldots)
        \int_{\Omega_j}
        \frac{\partial}{\partial x_i}
        u(x)\psi_k^j(x_j)dx_j
        dx_j^\wedge\\
        &+
        \sum_{i\neq j}
        \int_{\bigtimes_{k\neq j}\Omega_k}\frac{\partial^3}{\partial x_j^2
        \partial x_i}
        u(\ldots, x_{j-1},\cdot,
        x_{j+1},
        \ldots)
        \int_{\Omega_j}
        \frac{\partial^2}{\partial x_i\partial x_j}
        u(x_j^\vee)
        \frac{d}{dx_j}\psi_k^j(x_j)dx_j
        dx_j^\wedge.
    \end{align*}
    Thus, we can apply \cref{lemma:prbound} and conclude
    \begin{align*}
        \|P_r^j v\|_1\leq\sqrt{2r}D_j(r)\|v\|_0,\quad v\in H^1(\Omega_j).
    \end{align*}

    \textbf{III.} With the above we estimate further
    \begin{align*}
        &\|\left(P_r^1\otimes\cdots\otimes P_r^{j-1}
        \otimes\id_{j+1}\otimes\cdots\otimes\id_d\right)v\|_1^2\\
        &\leq
        \sum_{l\neq j}
        \|\left(P_r^1\otimes\cdots\otimes P_r^{j-1}
        \otimes\id_{j+1}\otimes\cdots\otimes\id_d\right)v\|_{e_l}^2\\
        &\leq
        \sum_{l\neq j}\prod_{\substack{k=1,\ldots,j-1,\\k\neq l}}
        \|P_r^k\|^2_{H^1\leftarrow L_2}
        \|v\|_{\bs l}^2\\
        &\leq 4r^{d-2}\max_{l\neq j}
        \prod_{\substack{k=1,\ldots,j-1,\\k\neq l}}D_k(r)^2\|v\|_1^2.
    \end{align*}

    \textbf{IV.} The last term in the error is simply
$
        \left\|u-\mc P_r^ju\right\|^2_{\mix, j}
        =\sum_{k=r+1}^\infty(\sigma_k^j)^2.
$
    Since the ordering $J\in S_d(I)$ can be chosen arbitrarily, the statement
    follows.
\end{proof}

%% file: ms.bbl
\begin{thebibliography}{1}

\bibitem{ANSVD}
{\sc Ali, M., and Nouy, A.}
\newblock {S}ingular {V}alue {D}ecomposition in {S}obolev {S}paces: {P}art {I}.
\newblock {\em Z. Anal. Anwend.\/} (2020), to appear.

\bibitem{AU}
{\sc {Ali}, M., and {Urban}, K.}
\newblock {HT-AWGM: A Hierarchical Tucker-Adaptive Wavelet Galerkin Method for
  High Dimensional Elliptic Problems}.
\newblock {\em ArXiv e-prints\/} (May 2018).

\bibitem{BDL2}
{\sc Bachmayr, M., and Dahmen, W.}
\newblock Adaptive low-rank methods for problems on {S}obolev spaces with error
  control in {${\rm L}_2$}.
\newblock {\em ESAIM Math. Model. Numer. Anal. 50}, 4 (2016), 1107--1136.

\bibitem{BDH1}
{\sc Bachmayr, M., and Dahmen, W.}
\newblock Adaptive low-rank methods: problems on {S}obolev spaces.
\newblock {\em SIAM J. Numer. Anal. 54}, 2 (2016), 744--796.

\bibitem{HB}
{\sc Hackbusch, W.}
\newblock {\em Tensor spaces and numerical tensor calculus}, vol.~42 of {\em
  Springer Series in Computational Mathematics}.
\newblock Springer, Heidelberg, 2012.

\bibitem{HY}
{\sc Yserentant, H.}
\newblock {\em Regularity and approximability of electronic wave functions},
  vol.~2000 of {\em Lecture Notes in Mathematics}.
\newblock Springer-Verlag, Berlin, 2010.

\end{thebibliography}
